\documentclass[a4paper,10pt]{amsart}
\pdfoutput=1
\usepackage[latin1]{inputenc}
\usepackage{amsmath,amssymb}
\usepackage{latexsym}
\usepackage[english]{babel}

\newcommand {\NN}  {{\mathbb N}}
\newcommand {\ZZ}  {{\mathbb Z}}
\newcommand{\set}[1]{\left\{#1\right\}}
\newcommand{\comment}[1]{}

\usepackage[paper=a4paper,left=30mm,right=20mm,top=25mm,bottom=30mm]{geometry}

\usepackage{amsthm}

\newcommand {\CC}  {{\mathbb C}}

\newtheorem{thm}{Theorem}

\newtheorem{cor}[thm]{Corollary}

\newtheorem{lem}[thm]{Lemma}

\newtheorem{prp}[thm]{Proposition}
\newtheorem{proposition}[thm]{Proposition}
\newtheorem{rem}[thm]{Remark}
\newtheorem{ex}[thm]{Example}

\newtheorem*{theorem*}{Theorem}
\newtheorem*{proposition*}{Proposition}
\newtheorem*{lemma*}{Lemma}
\newtheorem*{corollary*}{Corollary}
\newtheorem*{remark*}{Remark}
\newtheorem*{example*}{Example}

\title[Penney's algorithm]{A remark on Penney's algorithm}

\author[H.~Brunotte]{Horst~Brunotte}
\address{Haus-Endt-Stra{\ss}e 88 \\ D-40593 D\"usseldorf, GERMANY}
\email{brunoth@web.de}
\date{\today}

\keywords{integer representation,  canonical  representation, radix representation}

\subjclass[2020]{11A63, 11A67, 11C08, 11R11}

\begin{document}

\begin{abstract}
Based on the well-known algorithm of W. Penney \cite{penney} we determine the set of lengths of the canonical representation
of   integers with respect to the trinomial $X^{2 m} + 2X^m + 2$. 
\end{abstract}

 \maketitle
 
W.~Penney~\cite{penney} used $-1 +i$ as a basis of representing complex numbers\footnote{$\NN$ is the set of positive  rational integers and $\NN_0= \NN \cup \set{0}$.}  
 $$a+ b i       \qquad \qquad (a,b \in \{ k /  2^n  \;: \; k\in \ZZ, \; n\in \NN_0\})$$
by writing $a, b$ as
 $$ \sum_{j=0}^k c_j \bigl(  (-1 +i)^4 \bigr)^j \qquad \qquad (c_0\ldots, c_k \in \{0,1,2,3\}).$$
 Based on his algorithm we here consider the so-called canonical representation of integers with respect to the minimal 
 polynomial of $-1 +i$ and straightforwardly extend this method  to particular integral   trinomials of even degrees.
 
 Canonical number systems are  natural generalizations
of radix  representations of ordinary integers %(introduced by {\sc V. Gr\"unwald} \cite{Gr})
to algebraic integers. Given a monic integer polynomial $p$ with $|p(0)| >1$ we say that $z\in \ZZ$ admits a $p$-canonical representation if there exist  
a positive integer ${\rm \ell} $ and
  $$u_0,\ldots, u_{{\rm \ell} -1} \in D_p:= \set{0,\ldots, |p(0)| -1} \subset \NN_0$$ such that 
$$z \equiv \sum_{j=0}^{ {\rm \ell}-1  }u_j X^j \pmod p\,,$$
and we denote by $Z_p$ the set of $p$-canonically representable integers. It is well-known that this representation (if it exists) is unique (e.g., see \cite{kirthu}).
If ${\rm \ell}$ is minimal with the property above then  $${\rm \ell}_p (z):= {\rm \ell}$$ is called the length of the \mbox{$p$-canonical} representation of $z$,
and we shortly write
$$ z = (u_ { {\rm \ell}-1}  \cdots   u_ {0})_p\,.$$
For a given integer $c$ with  
\begin{equation}\label{cpass}
 c >  | p(0)|  \qquad \text{ and } \qquad \  | p(0)|, \ldots, c  - 1 \in  Z_p
\end{equation}
we define the function 
$$\lambda_{p , - c} :  \ZZ  \rightarrow   \NN$$ 
as follows. For $z \in \ZZ$ we determine  its $- c$-representation 
$$z=( v_k \cdots v_0)_{-c} \qquad \qquad  (v_0,\ldots, v_k \in \{0, 1, \ldots, c - 1\})$$ and set
$$ \lambda_{p, - c} (z) : = {\rm \ell}_p  (v_k)\,.$$
Note that the right hand side is well-defined because we actually have  $$\{0, 1, \ldots, c - 1\}  \subseteq Z_p\,.$$
Using this function we can state our main result.

\begin{thm}\label{ppen}
Let $$p := X^2 + 2X + 2\,.$$
\begin{enumerate}
\item  
${\rm \ell}_p (z) = 4 \bigl({\rm \ell}_{-4}(z)  -1\bigr) + \lambda_{p, -4} (z)  \qquad  \qquad (z\in   \ZZ)$
\item ${\rm \ell}_p (z) \equiv 0, 1 \pmod 4 \qquad  \qquad (z\in \ZZ)$\\
and we have $${\rm \ell}_p ( \ZZ) =\set{a(n)  \;  :  \; n \in \NN},$$
where the sequence  $a(n)_{n\in \NN}$ specifies the non-negative integers congruent to $0$ or $1$ modulo~$4$ and  is given by 
$$ a(n)= a(n-1) + (-1)^n +2\qquad \qquad (n \in \NN)$$
with $a(0)=0$ (see \cite[A042948]{sloeis}).\\
\item For $n, m\in \NN$ we have $$ {\rm \ell}_p  (n) \neq  {\rm \ell}_p  (-m)   \,.$$
\item Fix  $ {\rm \ell} \in \NN$. If $\rm \ell$ is odd and $n \in \NN$ maximal with the property
$$ {\rm \ell}_{-4}  ( n) =  {\rm \ell}$$
then we have 
$$\lambda_{p, -4} (n)= 4,   \quad \lambda_{p, -4} (n+1)= 1,   \quad   {\rm \ell}_p  (n) \equiv 0 \pmod 4 $$ and $$ {\rm \ell}_p  (n+1) =   {\rm \ell}_p  (n) + 5\,.$$
If $\rm \ell$ is even and  $n \in \NN$ least with the property
$$ {\rm \ell}_{-4}  (- n) =  {\rm \ell}$$
then we analogously have 
$$\lambda_{p, -4} (-n)= 4,   \quad \lambda_{p, -4} (-(n+1))= 1,   \quad   {\rm \ell}_p  (-n) \equiv 0 \pmod 4 $$
 and $$ {\rm \ell}_p  (-(n+1)) =   {\rm \ell}_p  (- n) + 5\,.$$
\item The subsequence of ${\rm \ell}_p ( \ZZ) $ which describes the lengths of the $p$-representations of the nonnegative (negative, respectively) integers is given by the sequences of consecutive pairs
$$\bigl(a(4n-3), \, a(4n-2)\big)_{n\in \NN} $$
and
$$\bigl(a(4n-1),  \, a(4n)\big)_{n\in \NN}  \text{, respectively. }$$
\item If  $n < m$ are  positive  integers then we have
$${\rm \ell}_p (m) = {\rm \ell}_p (n) \qquad \text{ or  } \qquad {\rm \ell}_p (m) \ge  {\rm \ell}_p (n) + 3$$
and
$${\rm \ell}_p (-m) = {\rm \ell}_p (-n) \qquad \text{ or  } \qquad {\rm \ell}_p (-m) \ge  {\rm \ell}_p (-n) + 3\,.$$
\item  $ -2 \le \lambda_{p, -4} (x) + \lambda_{p, -4} (y) - \lambda_{p, -4} (x y) \le 7 \qquad \qquad (x,y \in \ZZ),$ \\
and in both cases equality is possible.
\item For   $x, y \in \ZZ$  we have
$${\rm \ell}_p (x+y)  \le   {\rm \ell}_p (x) +  {\rm \ell}_p (y) + 2 $$
and
$${\rm \ell}_p (x y) \le   {\rm \ell}_p (x)  +  {\rm \ell}_p (y)  + 10\,.$$  
\item Let $z = (\delta_{ {\rm \ell}_p (z)-1}\cdots \delta_{0})_p \in \ZZ$ and define
$$s_{k+1}(z)= \frac1{2} \bigl(s_{k-1}(z)+ s_{k-2}(z) \bigr) \qquad \qquad (k\ge 2)$$ with
$$ s_{0}(z)=z, \; s_{1}(z)=0 \quad \text{ and } \quad s_{2}(z)= \frac1{2}  z \,.$$
Then there exists some $K\in \NN$ such that  $s_{K}(z)$ is an even integer,   
$$s_{k}(z)= s_{K}(z) \qquad \qquad (k\ge K)$$
and the sum of digits of $z$ is 
$$\sum_{i=0}^{  {\rm \ell}_p (z) -1} \delta_{i}= z- \frac5{2} s_K (z) \,.$$
\end{enumerate}
\end{thm}

Our result above immediately delivers analogous statements for trinomials of higher degrees. 

\begin{cor}\label{corp2m}
For  $$P= X^{2m} + 2   X^m+  2   \in  \ZZ[X]  \qquad \qquad (m \in \NN)$$
we have
$${\rm \ell}_P (\ZZ)=  \set{m \bigl(a(n) -1\bigr)+  1 \; : \; n\in \NN},$$
where the sequence $a(n)_{n\in \NN}$ is given in Theorem \ref{ppen}.
\end{cor}
\begin{proof}
In view of $$P(X) = p(X^m)$$
with $p$ as in Theorem \ref{ppen}
we have
$${\rm \ell}_P (z)= m ({\rm \ell}_p (z) - 1) + 1 \qquad \qquad (z\in \ZZ)$$
by Proposition \ref{pqrrepr} below, and then our claim drops out from Theorem \ref{ppen}.
\end{proof}

A simple example illustrates this result.

\begin{ex}\label{excorp2m}
Setting   $m=2$ in  Corollary \ref{corp2m}  we obtain
$${\rm \ell}_{X^{4} + 2   X^2+  2} (\ZZ)=  \set{2 a(n) -1 \; : \; n\in \NN},$$
and we have
$${\rm \ell}_{X^{4} + 2   X^2+  2} ( \ZZ) =  \set{b(n)  \;  :  \; n \in \NN_0},$$ where the sequence  $b(n)_{n\in \NN}$  
lists the positive integers  congruent to $1$ or $7$ modulo $ 8$  (see \cite[A047522]{sloeis}) and is given by
$$ b(n)=   \sqrt{8 (c(n+1)  -1)+1}  \qquad \qquad \qquad (n \in \NN_0)$$
with $c(n)_{n\in \NN}$ presented in  \cite[A014494]{sloeis}. Exploiting the remarks on this sequence we can write
$$c (n)  = \frac1{2}\Bigl(  4 n^2 + (-1)^n  (2 n - 1) - 4 n +3 \Bigr) \qquad \qquad (n\in \NN).$$
\end{ex}

Let us now prepare the proof of our theorem by several auxiliary results.

\begin{lem}\label{fsumdc}
Let  $n, d  \in  \NN$, $D$ a subset of a ring,  $0\in D$ and 
$f_0, \ldots, f_n \in D [X]$. If  $f_n \ne 0$ and\footnote{We use the convention $\deg (0) = - \infty$.}
  $$\deg (f_j)  < d  \qquad \qquad (j = 0,\ldots,n)$$
  then we have $$f:=  \sum_{j=0}^n f_j  X^{ j d } \in D [X] \qquad \text{ and } \qquad \deg (f)  = nd  + \deg (f_n)  \,.$$ 
\end{lem}
\begin{proof}
This can easily be checked.
\end{proof}
In the following we denote  by $\Omega_f$ the set of roots of the polynomial $f \in \CC[X]$.
\begin{lem}\label{npgrho}
Let $p \in \ZZ  [X]$ be monic  with $| p(0)| > 1$. 
\begin{enumerate}
\item Let $q\in \ZZ \setminus \{ 0 \}, q p(0) \in Z_p$ and $\ell  = {\ell}_p (z)$. Then %$\ell  >1$ and 
there exist
 $ v_1, \ldots,  v_{{\ell}  -1} \in D_p$ such that 
$$q p(0) +r = ( v_{{\ell}  -1} \cdots v_1 r)_p \qquad \qquad (r\in D_p).$$
\item  Let $z \in \ZZ$ and assume that all roots of $p$ are simple. If  $g\in D_p [X]$ with 
\begin{equation}\label{eqpdivgrho}
g(\rho) = z \qquad \qquad (\rho \in \Omega_p)
\end{equation}
then $g$ is the $p$-canonical representative of $z$.
\end{enumerate}
\end{lem}
\begin{proof}
This is immediately verified. 
\end{proof}

Now we can present  our main tool based on the arguments of \cite{penney}.  

\begin{lem}\label{prp}      
Let  $c\in \NN $ and $p$ be a monic  integer  polynomial with only simple roots and    
$$|p(0)|> 1 \quad \text{ and } \quad  q |p(0)|  \in  Z_p  \qquad (q\in \NN, \;  q \le (c - 1) / |p(0)|) \,.$$ 
Further suppose that there is some $d\in \NN$ such that  $p$ divides $X^d + c$ and 
$$d> \deg (p) \qquad \text{ and } \qquad	d \ge   \max \set{{\rm \ell}_p  (i) \;:\; i\in   \{0,1,\ldots, c -1\}}.$$
Then every $z\in \ZZ$ is $p$-canonically representable, its $p$-representative 
can easily be deduced from its  $-c$-representative  and we have
$${\rm \ell}_p  (z) = d \bigl({\rm \ell}_{-c}  (z)-1 \bigr) + \lambda_{p, -c} (z).$$
\end{lem}
\begin{proof}
Obviously, we have 
\begin{equation}\label{rhodc}
\rho^d=-c  \qquad \qquad (\rho \in \Omega_p),
\end{equation}
and  since $d$ exceeds $\deg (p) $  we have $$c > |p(0)|\,.$$ 
\comment{Consider f p = X^d +c, thus \deg(f) > 0 and $|f(0)|\ge 1$. The assumption $|f(0)| = 1$ yields some
\sigma \in \Omega_f with |\sigma| \le 1 which contradicts \sigma^d = - c.}By our prerequisites and Lemma \ref{npgrho} we convince ourselves that $$ D:= \{0,1,\ldots, c -1\} \subseteq Z_p\,,$$ thus for each $i\in D$  there exist 
$$u_{d -1}^{(i)}, \ldots, u_0^{(i)} \in D_p $$ (possibly with some  leading $0$'s)  such that 
for
$$h_i :=\sum_{j=0}^{d-1} u_j^{(i)} X^j \in D_p [X]  $$
we have
\begin{equation}\label{hirhoi}
\deg (h_i) < d\;,  \quad  h_i \equiv i \pmod p \qquad \text{ and } \qquad h_i (\rho) =  i \qquad \qquad (\rho \in \Omega_p). 
\end{equation}
For $z \in D_p$ we have $${\rm \ell}_p  (z)  = \lambda_{p,-c} (z)=1$$ and our claim is trivial.  Now let 
$z  \in \ZZ \setminus  D_p, \; \ell :=\ell_{-c} (z)$ and $$z=( v_{\ell-1} \cdots v_0)_{-c} \qquad \qquad 
 (v_0,\ldots,  v_{\ell-1} \in \{0, 1, \ldots, c - 1\}),$$ thus 
 \begin{equation}\label{lzpen} 
v_{\ell-1} \ne 0,  \;  g:= \sum_{i=0}^{\ell-1} v_i X^i \in D [X], \;  \deg (g)= \ell -1,   \;	z \equiv g \pmod {(X+c)}, \;
 \lambda_{p,-c} (z)  = {\rm \ell}_p  (v_{\ell-1} )\,.
\end{equation}
Exploiting  Lemma \ref{fsumdc} we have  
 \begin{equation}\label{gproplzd} 
h_{v_{\ell-1}} \ne 0,\;  \qquad G:= \sum_{i=0}^{\ell-1} h_{v_i} X^{ i d }\in D_p [X] \;  \; \text{ and } \quad \deg (G) = \deg (h_{v_{\ell-1}}) + (\ell-1) d\,.
\end{equation}
Using \eqref{rhodc} and \eqref{hirhoi} we have 
$$G (\rho) = \sum_{i=0}^{\ell-1} h_{v_i} (\rho)  (\rho^d) ^{i }= \sum_{i=0}^{\ell-1} v_i  (- c) ^{i }= g(-c) = z 
\qquad \qquad (\rho \in \Omega_p)\,,$$
hence
$$G \equiv z \pmod p $$
by Lemma \ref{npgrho}. Thus $G$ is the $p$-canonical representative of $z$ and  applying  \eqref{gproplzd} 
and \eqref{lzpen} we find 
\begin{eqnarray*}  
{\rm \ell}_p (z) & = & \deg (G) + 1=\deg (h_{v_{\ell-1}}) + (\ell -1) d + 1=\ell_p( {v_{\ell-1}}) -1 +(\ell -1) d + 1\\
&=& \lambda_{p, -c} (z) + d (\ell -1)   \,.
\end{eqnarray*}
\end{proof}For the sake of completeness we  collect some  well-known facts of the canonical representation of integers with a negative integer base.
\begin{prp}\label{tr-brept} 
Let $b, n\in \NN$ with $b>1$.
\begin{enumerate}\item
$\ell_{-b} (n)$ is odd  and  $\ell_{-b} (-n )$ is even.
\item $\ell_{-b} (\ZZ)=\NN$
\item $\ell_{-b}$ is increasing on $\NN_0$ and  decreasing on $- \NN$. 
\item $\ell_{-b} (n+1)=\ell_{-b} (n) \quad \text{ or } \quad   \ell_{-b} (n+1)=\ell_{-b} (n) + 2 $
\item $\ell_{-b} (- (n+1))=\ell_{-b} (- n) \quad \text{ or } \quad   \ell_{-b} (-(n+1))= \ell_{-b} (- n) + 2 $
\item For $k\in \NN_0$ the largest positive integer of $-b$-length $2k+1$ is 
$$n= ((b-1) 0 \ldots 0 (b-1))_{-b}= \frac{b^{2(k+1)}-1}{b+1} \,,$$ 
and
$$n+1=\underbrace{ (1 (b-1) 0 \cdots 0 (b-1)0) }_{2(k+1)+ 1}= \frac{b}{b+1}(b^{2k+1}+1)$$
is the least positive integer of $-b$-length $2k+3$.
For $k\in \NN$ the least negative integer of $-b$-length $2k$ is 
$$-n = ( (b -1) 0 \ldots  (b -1) 0)_{-b}=- \frac{b}{b+1}(b^{2k}-1) \,, $$
and  
$$-(n+1)=\underbrace{ (1 (b-1) 0 \cdots  (b-1) 0  (b - 1) )}_{2(k+1)}= - \frac{1}{b+1}(b^{2k+1}+1) $$
is the largest negative integer of $-b$-length $2(k+1)$.
\item For  $x,y \in \ZZ$ we have
$${\rm \ell}_{-b}  ( x+y) \le  \max \set{   {\rm \ell}_{-b}  (x) ,  {\rm \ell}_{-b}  (y)} + 1\,,$$
and there exists $e \in \set{-3, -1 , 1}$ such that
$${\rm \ell}_{-b}  ( x y) = {\rm \ell}_{-b}  (x)  + {\rm \ell}_{-b}  (y)  + e\,.$$
\end{enumerate}
\end{prp}
\begin{proof}
(i) E.g. see \cite[Proposition 3.1]{frolai11}.\\
(ii), (vi)  Obvious.\\
(iii) E.g. see \cite[Lemma 5.5]{hbdsgr}.\\
(iv), (v) Clear by (i)  and (iii).\\
(vi) This can straightforwardly be verified.\\
(vii)  We set  $k:={\rm \ell}  ( x)$ and  $m:={\rm \ell}  ( y)$  and assume $x \le  y$. In case $x>0$ we have $k\le m$ and  
\cite[Proposition 5.3]{hbdsgr} yields 
$$x \le \frac{ b^{k+1} -1  }{b+1} \qquad \text{ and } \qquad y \le \frac{ b^{m+1} -1  }{b+1}\,,$$
thus
$$x +y \le \frac{ b^{m+2} -1  }{b+1} \,,$$
and then
$${\rm \ell}  ( x+y) \le  m + 1\,.$$
Now we consider the case $x<0$. If $x+y\ge 0$ we see
$$0\le x+y < y$$
and  (ii) yields 
$${\rm \ell}  ( x+y) \le  {\rm \ell}  (y)\,,$$
and our claim drops out. Finally, we consider  $x+y <  0$. If $y > 0$ then we have 
$$x < x+y  < 0 $$ and (ii) implies
$${\rm \ell}  ( x+y) \le  {\rm \ell}  (x) \,,$$
and similarly the case $y<0$ is settled. For the second claim see  \cite[Proposition~5.3]{hbdsgr}.
\end{proof}

After these preparations we are now in a position  to  prove our main result. For convenience we write
$$\lambda:= \lambda_{p, -4}.$$
(i) We immediately check  
$$(X^2 - 2X +2) \cdot p= X^4 +4$$ and 
\begin{equation}\label{eqpncng}
0=(0)_p, \;  1=(1)_p, \;  2=(1 1 0 0)_p, \;  3=(1 1 0 1)_p\,,
\end{equation}
thus  
\begin{equation}\label{eqpnn0}
 \lambda (\ZZ) = \{1,4\}.
 \end{equation}
In view of 
$$2\in Z_p \qquad \text{ and } \qquad 4 \ge {\rm \ell}_p (j) \qquad (j=0,\ldots,3),$$ 
an application of Lemma \ref{prp} with $$c=d=4$$ yields our claim.\\
 (ii)  By (i) and the definition of the sequence $(a_n)$ we have 
\begin{equation}\label{eqeqla}
{\rm \ell}_p ( \ZZ) \subseteq \set{a(n)  \; :  \; n \in \NN}.
\end{equation}
To show equality we  convince ourselves that the sequence $(a(n))_{n\in\NN}$ can also be written in the form
$$8 k+1, 8k + 4, 8k +5, 8k +8 ,\ldots \qquad \qquad (k =0, 1, 2, 3, \ldots).$$
Fix $k\in \NN_0$. First, choose $n \in \NN$ minimal with
 \begin{equation}\label{eqpncn+}
 {\rm \ell}_{-4} (n)= 2k+1\,,
\end{equation}
 thus by  Proposition \ref{tr-brept} and \eqref{eqpncng}
 $$\lambda (n)= {\rm \ell}_p ( 1)=1$$
 and then by (i) and \eqref{eqpncn+}
 $${\rm \ell}_p (n ) = 4 (2k+1 -1)+1= 8k+1\,.$$
 Second, choose $n \in \NN$ maximal with \eqref{eqpncn+}, thus analogously as before
 $$\lambda (n)= {\rm \ell}_p ( 3)=4$$
and further  
 $${\rm \ell}_p (n ) = 4 (2k+1 -1)+4= 8k+4\,.$$
Third, we choose $n \in \NN$ minimal with
\begin{equation}\label{eqpncn-}
 {\rm \ell}_{-4} (- n)= 2(k+1)\,,
\end{equation}
thus by Proposition \ref{tr-brept}
$$\lambda (-n)= {\rm \ell}_p (1)=1$$
and then
 $${\rm \ell}_p (-n ) = 4 (2(k+1) -1)+1= 8k+5\,.$$
Fourth,  we choose $n \in \NN$ maximal with \eqref{eqpncn-} and deduce 
 $${\rm \ell}_p (-n ) =  8k+8\,.$$
 Thus, equality in  \eqref{eqeqla} is clear.\\
  (iii) The assumption of equality yields
$$ 4 ({\rm \ell}_{-4} (n)- {\rm \ell}_{-4} (- m)) =  \lambda (-m) - \lambda (n) \in \set{-3,0,3}$$
by (i). But in view of (ii) this is impossible because Proposition \ref{tr-brept} 
   yields
$$ {\rm \ell}_{-4} (n) \ne  {\rm \ell}_{-4} (- m)\,. $$
(iv) Clear by (i) and Proposition \ref{tr-brept}.\\
(v)  Similarly as in the proof of (ii) we immediately verify that
$$a(4n-3) \quad \text{ and } \quad a(4n-2) $$
are the consecutive $p$-lengths of positive integers (with difference $3$), 
and the elements
$$a(4n-1) \quad \text{ and } \qquad a(4n)$$
are the consecutive $p$-lengths of negative integers (also with difference $3$).
Each pair of consecutive $p$-lengths of positive integers is followed by a pair of consecutive $p$-lengths of negative integers. 
\\
(vi) This can immediately be checked by the definition of the sequence  $a(n).$\\ 
(vii) The first claim is trivial. To show possible equality, we may consider   
$$ 4 =  (130)_{-4}, \; 5= (131)_{-4}, \;  20 = (230)_{-4}\,, $$
$$  2 = (2)_{-4},  \;  410 = (22222)_{-4}, \;  820 = (1303030)_{-4}$$
$$\lambda (1) = {\rm \ell}_p (1)=1, \;  \lambda (2) = {\rm \ell}_p (2) = 4\,,$$
hence 
 $$    \lambda (20) = \lambda (4) + \lambda (5) +2   \qquad \text{ and } \qquad \lambda (820) = \lambda (2) + \lambda (410) - 7\,.$$
 (viii) We may assume $${\rm \ell}_{-4} (x)  \le {\rm \ell}_{-4}  (y)\,,$$ and exploiting  (i) and   Proposition~\ref{tr-brept}   we obtain
\begin{eqnarray*} 
{\rm \ell}_p (x+y)  & =   & 4 ({\rm \ell}_{-4}   (x+y) -1) + \lambda (x+y) \le  4 \max \set{   {\rm \ell}_{-4}  (x) ,  {\rm \ell}_{-4}  (y)}  +4\le
4 {\rm \ell}_{-4}  (y) +4  \\
 & =  & 4 ( {\rm \ell}_{-4}  (y)  -1) +\lambda (y) -  \lambda (y)  + 4 = {\rm \ell}_p (y)+  {\rm \ell}_p (x)-4 ( {\rm \ell}_{-4}  (x)  - 1) -  \lambda (x)  -  \lambda (y)  + 4\\
& =  & {\rm \ell}_p (x) + {\rm \ell}_p (y) - 4  {\rm \ell}_{-4}  (x)     -  \lambda (x) -  \lambda (y) +8 \le {\rm \ell}_p (x) + {\rm \ell}_p (y) - 4 -2 +8  \\
& =  &  {\rm \ell}_p (x) + {\rm \ell}_p (y) + 2\,.
\end{eqnarray*}
Analogously we establish the second claim
and we leave the details to the reader. \\
(ix) Clear by  \cite[Proposition 2.2 and Corollary 2.3]{GraKirPro98} with $a=1$. 

\medskip

The  proof of our theorem is now completed, and we finish by rounding off the proof of Corollary~\ref{corp2m}
with  a brief glance at the  representation of integers by non-primitive 
polynomials\footnote{According to 
\cite{jankauskas} we say that a polynomial is primitive if it is not of the form $g(X^k)$ for some $k> 1$.}.

\begin{proposition}\label{pqrrepr}
Let $p  \in \ZZ  [X]$ be monic, $k >1$ and $P:=p(X^k)$.
\begin{enumerate}
\item $D_P= D_p$ %$\Nc_P=\Nc_p$
\item 
If  $g  \in D_p [X]$ canonically represents $z  \in \ZZ$ modulo $p$ then   
$g (X^k) $ canonically represents $z$ modulo $P$. 
In particular, we have $Z_p \subseteq Z_P$,  and
for $z\in Z_p$ we have $${\rm \ell}_P (z)= k ({\rm \ell}_p (z) - 1) + 1\,,$$   
thus 
$${\rm \ell}_P (z) \equiv  1 \pmod k\,.$$
Explicitly,  $z  = (u_{l-1} u_{l-2} \cdots u_0)_p $ implies
$$z = (u_{l-1} w u_{l-2} w\cdots w u_0)_P $$
with
$$w:=\underbrace{0\cdots 0 }_{k - 1} \,.$$
\end{enumerate}
\end{proposition}
\begin{proof}
(i) Obvious.\\  
(ii)  Let $t \in \ZZ[X]$ and  $n \in \NN_0$ with $$p t=\sum_{i=0}^{n} u_i X^i - z \qquad \qquad (u_0,\ldots,u_n \in  D_p),$$ hence
$$P(X) t(X^k)=p(X^k) t(X^k)=(p t)(X^k) =\sum_{i=0}^{n} u_i X^{k i} - z,$$
and this implies our first claim. Clearly,  we have
$$ {\rm \ell}_p  (z)= n+1\qquad  \text{ and } \qquad  {\rm \ell}_P (z)= k n + 1,$$
and this implies our second assertion.\\
\end{proof}

\begin{rem}
\begin{enumerate}\item
The applicability of our main tool is very restricted. For instance, consider
$$p= X^{2} + 4   X+  8   \in  \ZZ[X]$$
and  
$$ ( X^{2} - 4   X+  8)  \cdot p = X^4 + 64\,.$$
We easily verify\footnote{For instance, use the algorithm given in \cite{kirthu}.} 
$$8 = (1 3 4  0)_p, 16 = (1 2 0 0)_p, 24  =  (2 5 4 0)_p, 32  =   (2 4 0 0)_p, 40  = (3 7 4 0)_p , 48  = (3 6 0 0)_p, 56 = (1 4 7 0 1 4 0)_p\,,$$
thus in view of $${\rm \ell}_p (8q)=4 \quad  (q=1,\ldots,6) \qquad \text{ and } \qquad {\rm \ell}_p (8\cdot 7)=7$$
Lemma \ref{prp}  cannot be applied here.
\item Other aspects of  modified canonical number systems in $\ZZ [i]$  are thoroughly studied in \cite{BogChe13}.
\end{enumerate}
\end{rem}

\bibliographystyle{siam}
\def\cprime{$'$}

\end{document}